\documentclass{siamltex}
\usepackage{amsfonts}
\usepackage{amsmath}
\usepackage{amssymb}
\usepackage{array}
\usepackage{graphicx}
\usepackage{color}
\usepackage{stmaryrd}

\def\dx{\hspace{2pt}{\rm d}x}

\def\l2{_{L_2(\Omega)}}
\def\nl2{_{[L_2(\Omega)]^n}}
\def\eps{\epsilon}

 \def\XXint#1#2#3{{\setbox0=\hbox{$#1{#2#3}{\int}$} 
\vcenter{\hbox{$#2#3$}}\kern-.5\wd0}}

\def\T{\mathcal{T}}

\def\M{{\mathcal{M}}}

\def\V{{\mathbb{V}}}

\definecolor{darkred}{rgb}{.7,0,0}
\newcommand\Red[1]{\textcolor{black}{#1}}
\definecolor{green}{rgb}{0,0.7,0}

\definecolor{myblue}{rgb}{0,0,0.7}

\newcommand\modif[1]{\textcolor{black}{#1}}

\newcounter{saveeqn}

\newtheorem{remark}[theorem]{Remark}

\graphicspath{{Results/}}

\numberwithin{equation}{section}

\title{Convergence and optimality of higher-order adaptive finite element methods for eigenvalue clusters
}

\author{
  Andrea Bonito\thanks{Department of Mathematics, Texas A\&M University, College Station TX, 77843; email: {\tt bonito@math.tamu.edu}. 
  Partially supported by NSF Grant DMS-1254618.
}
\and
Alan Demlow\thanks{Department of Mathematics, Texas A\&M University, College Station TX, 77843; email: {\tt demlow@math.tamu.edu}.
Partially supported by NSF Grant DMS-1518925.
}
}

\begin{document}

\maketitle

\renewcommand{\thefootnote}{\fnsymbol{footnote}}

\begin{abstract}   Proofs of convergence of adaptive finite element methods for approximating eigenvalues and eigenfunctions of linear elliptic problems have been given in a several recent papers.  A key step in establishing such results for multiple and clustered eigenvalues was provided by Dai et. al. in \cite{DHZ15}, who proved convergence and optimality of AFEM for eigenvalues of multiplicity greater than one.  There it was shown that a theoretical (non-computable) error estimator for which standard convergence proofs apply is equivalent to a standard computable estimator on sufficiently fine grids.  
 In \cite{Gal15}, Gallistl used a similar tool to prove that a standard adaptive FEM for controlling eigenvalue clusters for the Laplacian using continuous piecewise linear finite element spaces converges with optimal rate.  When considering either higher-order finite element spaces or non-constant diffusion coefficients, however, the arguments of \cite{DHZ15} and \cite{Gal15} do not yield equivalence of the practical and theoretical estimators for clustered eigenvalues.  In this note we provide this missing key step, thus showing that standard adaptive FEM for clustered eigenvalues employing elements of arbitrary polynomial degree converge with optimal rate.  We additionally establish that a key user-defined input parameter in the AFEM, the bulk marking parameter, may be chosen entirely independently of the properties of the target eigenvalue cluster.  All of these results assume a fineness condition on the initial mesh in order to ensure that the nonlinearity is sufficiently resolved.    
\end{abstract}

\begin{keywords}
eigenvalue problems, spectral computations, a posteriori error estimates, adaptivity, optimality \end{keywords}

\begin{AM} 65N12, 65N15, 65N25, 65N30
\end{AM}

\pagestyle{myheadings}
\thispagestyle{plain}
%\markboth{ALAN DEMLOW}{HIGHER-ORDER FEM AND POINTWISE ESTIMATES ON SURFACES}

%------------------------------------------------------------------------------------------------------------------------------------------ 
\section{Introduction}  \label{sec:intro}
%------------------------------------------------------------------------------------------------------------------------------------------

There has been high interest in recent years in the development and analysis of adaptive finite element methods (AFEM) for approximating eigenvalues and eigenfunctions of elliptic operators.  We consider the following model eigenvalue problem:  Find $(u_j,\lambda_j) \in H_0^1(\Omega) \times \mathbb{R}$ such that $(u_j,u_j)=1$ and
\begin{align}
\label{eq1-1}
a(u_j,v)=\lambda_j (u_j,v), ~~v \in H_0^1(\Omega).
\end{align}
Here $\Omega \subset \mathbb{R}^d$, $d=2,3$, is a polyhedral domain, $a(u,v):=\int_\Omega \nabla u \cdot \nabla v \dx$ and $(u,v):=\int_\Omega uv \dx$.  There is then a sequence of eigenvalues $0<\lambda_1 <\lambda_2 \le \lambda_3 \le ....$ and corresponding $L_2$-orthonormal eigenfunctions $u_1, u_2,...$ satisfying \eqref{eq1-1}.     Given a nested sequence of adaptively generated simplicial meshes $\{\T_\ell \}_{\ell \ge 0}$ with associated finite element spaces $\{ \V_\ell\}_{\ell \ge 0}$ ($\V_\ell \subset H_0^1(\Omega)$), the corresponding discrete eigenvalue problem is:  Find $(u_{\ell,j}, \lambda_{\ell,j}) \in \V_\ell \times \mathbb{R}$ such that $(u_{\ell,j},u_{\ell,j})=1$ and
\begin{align}
\label{eq1-2}
a(u_{\ell,j}, v) = \lambda_{\ell,j} (u_{\ell,j}, v), ~ v \in \V_\ell.
\end{align}
We seek to approximate an eigenvalue cluster $\{\lambda_j\}_{j \in J}$ and associated invariant subspace $\mathbb W:={\rm span}\{u_j\}_{j \in J}$.  Our index set $J$ is given by $J:=\{n+1,...,n+N\}$ for some $n\geq 0$, $N \geq 1$.  
The corresponding discrete sets are $\{\lambda_{\ell, j}\}_{j \in J}$ and $\mathbb W_\ell = {\rm span} \{u_{\ell, j}\}_{j \in J}$.  

AFEM for eigenvalues are typically based on the standard loop
$$
\textsf{solve} \rightarrow \textsf{estimate} \rightarrow \textsf{mark} \rightarrow \textsf{refine}.
$$ 
To \textsf{estimate} the finite element error, AFEM employs local error indicators $\eta_\ell(T)^2 := \sum_{j \in J} \eta_\ell (T, u_{\ell,j}, \lambda_{\ell, j})^2$, where $\eta_\ell(T, u_{\ell,j}, \lambda_{\ell,j})$ is a standard residual error indicator for the residual $-\Delta u_{\ell,j} - \lambda_{\ell, j} u_{\ell, j}$.  Let $0<\theta \le 1$ be a given parameter.  $\eta$ is used in \textsf{mark} to select a smallest set $\M_\ell \subset \T_\ell$ satisfying the D\"orfler (bulk) \cite{dorfler} criterion
\begin{equation}\label{eq1-4}
\sum_{T \in \M_\ell} \eta_\ell(T)^2 \ge \theta \sum_{T \in \T_\ell} \eta_\ell(T)^2.
\end{equation}

Proofs of convergence and optimality of AFEM for approximating \eqref{eq1-1} have been given in several papers.  The first proof of optimality of AFEM for controlling simple eigenvalues and eigenfunctions was given in \cite{DXZ08}.  Other papers concerning convergence of AFEM for simple eigenvalues include \cite{GG09, CaGe11, GM11}.  The paper \cite{DHZ15} contains a proof of optimal convergence of standard AFEM for an eigenvalue with multiplicity greater than one, while \cite{Gal15} proves a similar result for clustered eigenvalues.  These papers mirror AFEM convergence theory for source problems (cf. \cite{CKNS08}) in that they first prove that the AFEM contracts at each step.  An optimal convergence rate dependent on membership of the eigenfunctions in suitable approximation classes is then obtained by standard methods.  All require that the maximum mesh diameter in the initial mesh be sufficiently small to suitably resolve the nonlinearity of the problem.  The behavior of AFEM for eigenvalues in the pre-asymptotic regime was studied in \cite{GMZ09}, where the authors proved plain convergence results (with no rates) starting from any initial mesh.  These results guarantee convergence of AFEM for general elliptic eigenproblems to some eigenpair, but not generically to the correct pair.  

The works of Dai et al. \cite{DHZ15} and Gallistl \cite{Gal15} are most relevant to ours. 
In \cite{DHZ15} the authors establish convergence of an AFEM for a multiple eigenvalue of a symmetric second-order linear elliptic operator for arbitrary-degree finite element spaces.  \modif{A similar result is stated for eigenvalue clusters, but not all steps of the proof are provided} \Red{and the asymptotic nature of and constants in the results arising from the  proof suggested in \cite{DHZ15} depend on spectral resolution {\it within the target cluster.}}  Approximation of eigenclusters of the Laplacian using piecewise linear elements is considered in \cite{Gal15}.  \Red{The framework of \cite{Gal15} is {\it cluster-robust}, that is, all constants and the asymptotic nature of the estimates depending only on {\it separation of the target cluster from the remainder of the spectrum}}.  This leaves open the question of \Red{cluster-robust} convergence results for AFEM for eigenvalue clusters using Lagrange spaces of arbitrary polynomial degree.  We fill this gap by showing that standard AFEM for eigenvalue clusters using polynomials of arbitrary degree also converge optimally.  While we consider only conforming simplicial meshes, we also provide a key step in extending such analysis to quadrilateral elements of any degree, meshes with hanging nodes, and discontinuous Galerkin methods; cf. \cite{BN10} for analysis of the source problem.  The analysis of \cite{Gal15} additionally does not immediately apply in the case of non-constant diffusion coefficients.  In contrast, our results extend as in \cite{DHZ15} to general symmetric second-order linear elliptic operators (see Remark \ref{rem:genops}).

 We briefly explain the difficulty which we resolve.   A key step in standard AFEM convergence proofs establishing a certain continuity between error indicators on adjacent mesh levels. In the case of multiple or clustered eigenvalues, the ordering and alignment of the discrete eigenfunctions may change between mesh levels even on fine meshes.  Thus $(u_{\ell, j}, \lambda_{\ell, j})$ and $(u_{\ell+1, j}, \lambda_{\ell+1, j})$ may not approximate the same eigenpair, making the comparison between $\eta_\ell (T, u_{\ell, j}, \lambda_{\ell, j})$ and $\eta_{\ell+1} (T, u_{\ell+1,j}, \lambda_{\ell+1, j})$ used in standard AFEM convergence proofs irrelevant.  A critical contribution was made in \cite{DHZ15}, where this problem was circumvented by first analyzing a theoretical (non-computable) AFEM based on error indicators $\mu_\ell(T)^2 := \sum_{j \in J} \mu_j(T, u_j, \lambda_j)^2$ aligned to the {\it fixed} continuous cluster.  $\mu_\ell$ may be viewed as an indicator for a ``pseudo-residual'' $-\Delta \Lambda_\ell u_j-\lambda_j P_\ell u_j$, where $\Lambda_\ell$ and $P_\ell$ are projections onto $\mathbb W_\ell$ defined later.  Given $\widetilde \theta \in (0,1]$, the marking  strategy consists of selecting the smallest set $\widetilde{\mathcal M_\ell} \subset \mathcal T_\ell$ of $\Omega$ with 
\begin{equation}\label{eq3-3}
\sum_{T \in \widetilde{\mathcal M_\ell}} \mu_\ell^2(T) \geq \widetilde \theta \sum_{T \in \mathcal T_\ell} \mu_\ell(T)^2,
\end{equation}
 Optimality is guaranteed (see \cite{Gal15}) provided  the initial mesh is fine enough and $\tilde \theta \leq \tilde \theta^*$ for some $\widetilde\theta^*\leq 1$ only depending on $\Omega$ and the initial subdivision $\mathcal T_0$.
To deduce optimal properties of the practical algorithm based on $\eta_\ell(T)$ instead of $\mu_\ell(T)$,  an elementwise equivalence of $\eta_{\ell}$ and $\mu_\ell$ is used to show that the corresponding AFEM are essentially equivalent.  
\modif{
The extension of the proofs given for this final step in \cite{DHZ15, Gal15} to the approximation of clustered eigenvalues using higher-degree finite element spaces is however unclear.
}
%The given proof for this final step in \cite{DHZ15, Gal15} is however invalid when higher-degree finite element spaces are used to approximate clustered eigenvalues.  
We provide (Lemma~\ref{lem:main}) a more direct proof of the equivalence between $\mu_\ell$ and $\eta_\ell$ which is valid for finite element spaces of arbitrary degree.   Extension of the AFEM convergence results for eigenvalue clusters of \cite{Gal15} to higher-degree finite element spaces is then immediate  (see Theorem~\ref{t:opt} and Corollary~\ref{c:opt}).  \modif{Our estimator equivalence argument could also be used with minimal modification to complete the proof of AFEM optimality for eigenvalue clusters given in \cite[Theorem 7.1]{DHZ15}.} \Red{It is however not immediately clear how to modify certain other estimates in \cite{DHZ15} in order to obtain cluster-robust results.}    

%There are two further advantages to our approach.  First, our choice of $\mu_\ell$ differs from that in \cite{DHZ15}, while \cite{Gal15} suggests another possible choice of $\mu_\ell$ along with the one we make.  Our analysis hinges on a careful choice of the pseudo-residual, and we thus seemingly clarify the best choice.  

\modif{Our approach also yields} equivalence constants between $\mu_\ell$ and $\eta_\ell$ which are independent of essential quantities on sufficiently refined meshes, leading to a final result which is more robust with respect to the properties of the target cluster.
In view of \eqref{eq1-4}
 and \eqref{eq3-3}, optimal AFEM require $\theta \leq \theta^*:= C \widetilde \theta^*$, where $C < 1$ is independent of $N$, $\lambda_{n+1}$, and $\lambda_{n+N}$ (see Remark~\ref{r:theta}).  In contrast, the analysis of \cite{Gal15} implies that a stricter (smaller) choice of $\theta$ may be necessary as $N$ and the ratio $\lambda_{n+N}/\lambda_{n+1}$ increase.  As numerical experiments confirm (\S \ref{sec:numerics}), this improvement is important from a practical standpoint as it establishes that no knowledge of cluster properties is needed in order to choose $\theta$ correctly.  It also potentially increases computational efficiency by confirming that $\theta$ may be chosen reasonably large even when computing large clusters. On the other hand, \cite{Gal15} implies that $\theta$ should be taken to be small as $N$ increases, thus yielding a large number of adaptive loops and decreasing computational efficiency.  
 
The remainder of the note is structured as follows.  In \S\ref{sec:prelims} we give further assumptions and preliminaries.  In \S\ref{sec:main_result} we prove our main result and briefly sketch its application in the proof of AFEM convergence for eigenvalue clusters.  A brief set of numerical experiments in \S\ref{sec:numerics} illustrates our theoretical results.  
Finally note that we do not give a full proof of eigenvalue convergence, as most of the analysis of \cite{Gal15} holds verbatim for higher-order elements given the equivalence of the theoretical and practical error estimators.  We also largely employ the notation of that work.  Familiarity with \cite{Gal15} is thus essential to the reader.

%------------------------------------------------------------------------------------------------------------------------------------------ 
\section{Preliminaries} \label{sec:prelims}
%------------------------------------------------------------------------------------------------------------------------------------------

%------------------------------------------------------------------------------------------------------------------------------------------ 
\subsection{Finite element meshes and spaces} \label{subsec:fespace}
%------------------------------------------------------------------------------------------------------------------------------------------

 Let $\Omega$ be a polyhedral domain and denote by $\mathbb{T}$ the set of all conforming refinements of $\T_0$ obtainable by the newest-vertex bisection algorithm typically used in AFEM or its generalization to higher dimensions; cf. \cite{BinevDahmenDevore04, Ste07, Ste08,NV12}.
 By construction all $\T \in \mathbb T$ are uniformly shape regular. 

Given $\mathcal T \in \mathbb T$, we denote by $\V(\mathcal T) \subset H_0^1(\Omega)$ the space of continuous piecewise polynomials of arbitrary but fixed degree $r$ subordinate to $\mathcal T$. For each $\ell \ge 0$, we abbreviate $\V_\ell:= \V(\mathcal T_\ell)$. 
The AFEM produce sequences of subdivisions $\{\T_{\ell}\}_{\ell \ge 0} \subset \mathbb{T}$ such that $\T_{\ell+1}$ is a refinement of $\T_\ell$.  Thus $\V_{\ell} \subset \V_{\ell+1}$.

We also define several projection operators.  Let $P_\ell:L_2(\Omega) \rightarrow \mathbb W_\ell$  be the $L_2$ projection onto $\mathbb W_\ell$, and let $G_\ell:H_0^1(\Omega) \rightarrow \V_\ell$ be the Ritz projection.  That is,
\begin{align}
\label{eq3-1}
(P_\ell u, v) = (u, v), ~v \in \mathbb W_\ell, ~~~~ a(G_\ell u,v) = a(u, v), ~ v \in \V_\ell.
\end{align}
In addition, we define $\Lambda_\ell := P_\ell \circ G_\ell$.  Finally, $P:L_2(\Omega) \rightarrow \mathbb \mathbb W$ is the $L_2$ projection onto the continuous invariant subspace $\mathbb W$ corresponding to the target eigenvalue cluster.

%------------------------------------------------------------------------------------------------------------------------------------------ 
\subsection{Error estimators and AFEM} \label{subsec:est_def}
%------------------------------------------------------------------------------------------------------------------------------------------

Given a discrete eigenpair $(u_{\ell, j}, \lambda_{\ell,j})$ with $j \in J$ and $T \in \mathcal T_\ell$, the computable local error indicator $\eta_\ell(T)$ is given by
\begin{align}
\label{eq1-3}
\eta_\ell(T) ^2:= \sum_{j \in J} \left( h_T^2 \|\Delta u_{\ell, j} + \lambda_{\ell, j} u_{\ell, j}\|_T^2 + h_T \|\llbracket \nabla u_{\ell, j} \rrbracket \|_{\partial T}^2 \right), \qquad T \in \mathcal T_\ell.
\end{align}
Here $\llbracket \cdot \rrbracket$ is the jump across the element boundary, $\| \cdot\|_D$ is the $L_2$ norm over $D$ and $h_T := \textrm{diam}(T)$.  Let $|||\cdot|||=\sqrt{a(\cdot, \cdot)}$ be the energy norm.  The computable error estimator $\eta_\ell^2 :=\sum_{T \in \T_\ell} \eta_\ell(T)^2$ reliably controls \modif{$\sum_{j \in J} |||u_{\ell, j} - \Lambda_\ell u_{j}|||^2$} up to higher-order terms which can be proved to be negligible under the assumption that $H_0 := \max_{T \in \T_0} h_T$ is sufficiently small.  

The theoretical local indicator $\mu_\ell(T)$ is given by
\begin{align}
\label{eq3-2}
\mu_\ell(T)^2 := \sum_{j \in J} \left( h_T^2 \|\lambda_j P_\ell u_j + \Delta \Lambda_\ell u_j\|_T^2 + h_T \|\llbracket \nabla \Lambda_\ell u_j\rrbracket \|_{\partial T}^2\right).
\end{align}
A slight modification of Proposition 4.1 of \cite{Gal15} shows that $\sum_{T \in \T_{\ell}} \mu(T)^2$ also bounds $|||u-\Lambda_\ell u|||^2$ up to higher-order terms which are negligible if $H_0$ is sufficiently small.  In \cite{Gal15} the volume term is simplified to $h_T^2 \|\lambda_j P_\ell u_j\|_T^2$ due to the use of piecewise linear finite element spaces, and the definition in the higher-order case is left open with our choice given as one possibility (cf. Remark 9 (f) of \cite{Gal15}).  Our definition of $\mu_\ell$ also differs from that given in \cite{DHZ15}.   \modif{The difference does not appear to be essential} \Red{for purposes of proving estimator equivalence.  However, certain other estimates needed to obtain a cluster-robust final result have thus far only been obtained using the choice of $\mu_\ell$ suggested in \cite{Gal15} and used here. }  

%Our analysis clarifies that it is important to define $\mu_\ell$ as we do in order to obtain a precise equivalence between the theoretical and practical estimators.  

%------------------------------------------------------------------------------------------------------------------------------------------ 
\subsection{Assumptions on the initial mesh} \label{subsec:cluster_properties}
%------------------------------------------------------------------------------------------------------------------------------------------

Proofs of optimal convergence rates for adaptive FEM require assumptions on the eigenspaces and their resolution by the initial mesh.  First, the continuous eigencluster must be separated from the remainder of the spectrum, i.e., $\lambda_n< \lambda_{n+1} \le ... \le \lambda_{n+N} < \lambda_{n+N+1}$ (here we take $\lambda_0=0$).  In addition, the discrete cluster respects this separation in that $\lambda_{\ell, n} < \lambda_{n+1}$ and $\lambda_{\ell, n+N} < \lambda_{n+N+1}$.  Let then $M_J:=\sup_{\T_\ell \in \mathbb{T}} \max_{j \in \{1, ..., {\rm dim} \V_\ell \} \setminus J} \max_{k \in J} \frac{\lambda_k}{|\lambda_{\ell, j} - \lambda_k|}$.  The implied requirement on the initial mesh may be more or less strict depending on the separation of the cluster from the remainder of the spectrum. 

Reliability of our a posteriori estimators is only guaranteed for sufficiently fine initial mesh $\T_0$.  %To make more precise the discussion in the preceding subsection, 
We have for example from Proposition 4.1 of \cite{Gal15} that 
\begin{align}
\label{eq3-30}
|||u_j-\Lambda_\ell u_j |||^2 \le C \left[\sum_{T \in \T_\ell} \mu_\ell(T)^2 + \lambda_j^2 (1+M_J)^2 H_0^{2s}|||u_j-\Lambda_\ell u_j|||^2\right].
\end{align}
Thus the estimator is reliable when $H_0$ is small enough to guarantee that 
$
C\lambda_j (1+M_J)^2 H_0^{2s} \le \frac{1}{2}.
$  
This assumption is a priori and cannot be rigorously verified computationally.  Any assumption on the resolution of eigenvalues {\it within} the cluster is however avoided.   Cf. \cite{Gal15, BO89, GMZ09} for more discussion.

%------------------------------------------------------------------------------------------------------------------------------------------ 
\section{Main result} \label{sec:main_result}
%------------------------------------------------------------------------------------------------------------------------------------------

%------------------------------------------------------------------------------------------------------------------------------------------ 
\subsection{Equivalence of theoretical and practical estimators} \label{subsec:equivalence}
%------------------------------------------------------------------------------------------------------------------------------------------

Below is our main result.  

\begin{lemma} [Estimator Equivalence]\label{lem:main}
Assume that \modif{$H_0=\max_{T \in \T_0} h_T$} is small enough so that 
\begin{align}
\label{eq3-9}
\max_{j \in J} \|u_j- \Lambda_\ell u_j\|_\Omega \le \sqrt{1+(2N)^{-1}} -1, ~~~j\in J.  
\end{align}
Then for $T \in \T_\ell$,
\begin{align}
\label{eq3-10}
\mu_\ell(T)^2 \le \frac{3}{2} \eta_\ell(T)^2 \le 3 \mu_\ell(T)^2.
\end{align}
\end{lemma}

\begin{proof}
Using \eqref{eq1-1}, \eqref{eq1-2}, and \eqref{eq3-1}, we compute
\begin{align}
\label{eq3-11}
\begin{aligned}
\lambda_j P_\ell u_j&=  \lambda_j \sum_{m \in J} (u_j, u_{\ell,m}) u_{\ell,m}   = \sum_{m \in J} a(u_j, u_{\ell,m}) u_{\ell,m} 
\\ & = \sum_{m \in J} a(G_\ell u_j, u_{\ell,m}) u_{\ell,m} = \sum_{m \in J} \lambda_{\ell,m} (G_\ell u_j, u_{\ell,m}) u_{\ell,m}
\\ &  = \sum_{m \in J} \lambda_{\ell,m} (\Lambda_\ell u_j, u_{\ell,m}) u_{\ell,m}.
\end{aligned}
\end{align}
Also,
\begin{align}
\label{eq3-12}
\Lambda_\ell u_j = \sum_{m \in J} (\Lambda_\ell u_j, u_{\ell,m}) u_{\ell,m},
\end{align}
Combining the above two equations yields
\begin{align}
\label{eq3-13}
\begin{aligned}
\lambda_j P_\ell u_j + \Delta \Lambda_\ell u_j & = \sum_{m \in J} (\Lambda_\ell u_j, u_{\ell,m}) [\lambda_{\ell,m} u_{\ell,m} + \Delta u_{\ell,m}],  
\\ \llbracket \nabla \Lambda_\ell u_j \rrbracket & = \sum_{m \in J} (\Lambda_\ell u_j, u_{\ell,m})\llbracket \nabla u_{\ell,m} \rrbracket .
\end{aligned}
\end{align}

We now define the vectors $V:=[\lambda_j P_\ell u_j+ \Delta \Lambda_\ell u_j]$, $V_\ell := [\lambda_{\ell,m} u_{\ell,m} + \Delta u_{\ell,m}]$, $W:=[\llbracket \nabla \Lambda_\ell u_j \rrbracket]$, and \modif{$W_\ell:=[\llbracket \nabla u_{\ell,j} \rrbracket]$}.  Here  $n+1 \le j,m \le n+N$.  These quantities vary with $x \in \Omega$ (for $V$ and $V_\ell$) or $x$ lying on the mesh skeleton (for $W$ and $W_\ell$), but this fact has little significance for the time being.  We also define the matrix $M:=[(\Lambda_\ell u_j, u_{\ell,m})]$, with $j$ the row and $m$ the column index.  The relationships \eqref{eq3-13} may now be written
\begin{align}
\label{eq3-14}
V=MV_\ell, ~~~W=M W_\ell.
\end{align}
Let $\|\cdot\|_2$ denote the operator $\ell_2(\mathbb{R}^N)$-norm and $|\cdot|$ the Euclidean length.  The proof of Lemma \ref{lem:main} is thus reduced to showing that 
\begin{align}
\label{eq3-14-a}
\|M\|_2 \le \sqrt{3/2}, ~~~~ \|M^{-1}\|_2 \le \sqrt{2}
\end{align}
under the stated hypotheses.  

We now employ several basic facts from linear algebra.  Letting $v \in \mathbb{R}^N$, we have that $\|Mv\|_2^2=v^{t} M^{t} M v \le \|M^{t} M\|_2 |v|^2$ and thus $\|M\|_2^2 \le \|M^{t} M\|_2$.  Because $M^{t} M$ is symmetric and positive, $\|M^{t} M\|_2$ is equal to the maximum eigenvalue of $M^{t} M$.  

A short computation shows that $B:=M M^{t}= [(\Lambda_\ell u_j, \Lambda_\ell u_m)]$.  A standard result from linear algebra is that $AC$ and $CA$ are isospectral for square $A, C$, so $B$ is isospectral with $M^{t} M$.  Therefore we have that $\|M^{t} M \|_2 = \|B\|_2$, and both quantities are equal to the maximum eigenvalue of $B$.  The matrix $B$ was analyzed extensively in the proof of Lemma 5.1 of \cite{Gal15}.  In particular, $B$ is nonsingular under the condition \eqref{eq3-9}, and by (5.2) and following of \cite{Gal15}, we have
\begin{equation}
\label{eq3-15}
\frac{2N-1}{2N}  \le B_{ii} \le \frac{2N+1}{2N} \qquad  \text{and} \qquad  \sum_{j \neq i} B_{ij}  \le \frac{N-1}{2N}.
\end{equation}
Gershgorin's theorem thus gives that the eigenvalues $\{\sigma_i\}$ of $B$ satisfy
\begin{align}
\label{eq3-16}
\frac{1}{2}=\frac{2N-1}{2N}-\frac{N-1}{2N} \le \sigma_i \le \frac{2N+1}{2N}+\frac{N-1}{2N} = \frac{3}{2}, ~1 \le i \le N.
\end{align}
Thus $\|M\|_2^2 \le \|M^{t} M\|_2 = \|M M^{t}\|_2\le \frac{3}{2}$, which is the first inequality in \eqref{eq3-14-a}.  
 
 The invertibility of $B$ guarantees the invertibility of $M$.   Computing as above yields $\|M^{-1}\|_2^2 \le \|(M^{-1})^{t} M^{-1}\|_2 =\|B^{-1}\|_2$.  Because $B$ is positive and diagonalizable, we have from \eqref{eq3-16} that $\|B^{-1}\|=\frac{1}{\min_{1 \le i \le N} \sigma_i} \le 2$, thus completing the proof of the second inequality in \eqref{eq3-14-a}.  
\end{proof}

\begin{remark}[{Mesh Fineness Assumption}] {
%We comment on the effect of the mesh fineness assumption \ref{eq3-9} and compare with the corresponding results in \cite{DHZ15, Gal15}.  
In \cite{Gal15} it is shown for piecewise linear elements that under the assumption \eqref{eq3-9}, 
\begin{align}
\label{eq3-17}
\mu_\ell(T)^2 \le N (\lambda_{n+N}/\lambda{n+1})^2 \eta_\ell(T)^2 \le (\lambda_{n+M}/\lambda_{n+1})^4 (2N^2 + 4N^3) \mu_\ell(T)^2.  
\end{align}
The first bound in \eqref{eq3-17} holds with no restriction on $H_0$.  Our proof thus yields an improved cluster-independent bound under the fineness assumption \eqref{eq3-9}.  \eqref{eq3-10} also gives an improved bound for $\eta_\ell$ in terms of $\mu_\ell$.  Significantly loosening the restriction \eqref{eq3-9} appears to be substantially more difficult.  $M$ is invertible for $\eps < \sqrt{1+N^{-1}}-1$, but the obtained upper bound for $\|M^{-1}\|_2$ degenerates as $\eps \uparrow \sqrt{1+N^{-1}}-1$.  The condition $\eps< \sqrt{1+N^{-1}}-1$ is not substantially weaker than \eqref{eq3-9}, and control over the equivalency constants is lost as $\eps \uparrow \sqrt{1+N^{-1}}-1$, so we retain \eqref{eq3-9} from \cite{Gal15}.  
}\end{remark}

\begin{remark}[{Alternate proof of estimator equivalence}]
In Lemma 3.3 of \cite{DHZ15} it is shown that if $\lambda_{n+1}=\lambda_{n+N}$ and $\V_\ell$ is of arbitrary polynomial degree,
\begin{align}
\label{eq3-17-a}
\tilde{\mu}_\ell (T)^2 \le N \eta_\ell (T)^2 \le N^2 \widetilde{C} \left ( \mu(T)^2 + \sum_{j \in J} \sum_{k \in J} |\lambda_{\ell, j} -\lambda_{\ell, k}|^2 \modif{h_T^2} \|u_{\ell, k}\|_T^2 \right ),
\end{align}
where $\widetilde{C} \rightarrow 1$ as $H_0 \rightarrow 0$ and $\tilde{\mu}_\ell$ is similar to $\mu_\ell$.  In our analysis the obtained bounds are independent of $N$ and the double sum term on the right is completely absent.  The absence of the double sum term on the right is important, as \modif{it is of higher order for multiple but not for clustered eigenvalues}; cf. Remark 9 (f) of \cite{Gal15}.  \Red{Usable results for clusters may be obtained by summing \eqref{eq3-17-a} over each multiple eigenvalue in a cluster, but the constants in and asymptotic nature of the resulting estimate then are not cluster-robust.}
\end{remark}

\begin{remark}[{General symmetric operators}] \label{rem:genops} { In \cite{DHZ15} AFEM optimality is analyzed for eigenvalues of arbitrary second-order symmetric operators $-\nabla \cdot (A \nabla u) + cu$. 
% Extension of the arguments of \cite{Gal15} used to obtain equivalence of $\eta_\ell$ and $\mu_\ell$ to the case of nonconstant diffusion coefficients is not clear.  
Letting $\mu_\ell(T)^2 = h_T^2 \|\lambda_j P_\ell u_j + \nabla \cdot (A \nabla \Lambda_\ell u_j) - c \Lambda_\ell u_j\|_T^2 + h_T \|\llbracket A \nabla \Lambda_\ell u_j \rrbracket \|_{\partial T}^2$, our equivalence analysis extends immediately to the general case.  We consider only $A=I$, $c=0$ because analyzing optimality in the general case involves consideration of data oscillation, which would make our presentation much more involved; cf \cite{DHZ15}.  
}
\end{remark}  

%------------------------------------------------------------------------------------------------------------------------------------------ 
\subsection{AFEM optimality} \label{subsec:convergence} 
%------------------------------------------------------------------------------------------------------------------------------------------
Let
\modif{
$$
\mathcal A_\sigma^J := \left\lbrace \mathbf v = (v_1,...,v_J) \in H_0^1(\Omega)^J: |\mathbf v|_{\mathcal  A_\sigma^J}   <\infty \right\rbrace, 
$$
where $0<\sigma \leq r/d$ and
$$
 |\mathbf v|_{\mathcal  A_\sigma^J}:= \sup_{m \in \mathbb{N}} m^\sigma \inf_{\T \in \mathbb{T}, \#\T -\# \T_0 \le m} \inf_{\boldsymbol\phi =(\phi_1,...,\phi_J) \in \mathbb V(\mathcal T)^J} \left( \sum_{j=1}^J||| v_j - \phi_j |||^2\right)^{1/2}.
$$
}
Given Lemma \ref{lem:main}, the analysis of \cite{Gal15} mostly holds verbatim for higher-order finite element spaces.  In particular, the following counterpart to Theorem 3.1 of \cite{Gal15} holds.  

\begin{theorem}[Quasi-Optimal Approximation of Eigenspaces]\label{t:opt}
Assume that the bulk parameter $\theta<1$ and initial mesh size $H_0$ are sufficiently small.
Assume that for some $0<\sigma \leq r/d$,  \modif{$\mathbf u:=(u_{n+1},...,u_{n+N}) \in \mathcal A_\sigma^J$}.  Then there exists $C>0$ depending possibly on $\T_0$ but independent of other essential quantities such that
\begin{align}
\label{eq3-19}
\left(\sum_{j \in J} |||u_j-\Lambda_\ell u_j|||^2\right)^{\frac12} \le C (\# \T_{\ell} - \# \T_0)^{-\sigma} \modif{ |\mathbf u |_{\mathcal A_\sigma^J}}.  
\end{align}
\end{theorem}

Several remarks are in order.
\modif{
\begin{remark}[Approximation classes for clusters]\label{r:cluster}
We consider approximation classes  $\mathcal A_\sigma^J$ for the entire cluster of eigenfunctions, thereby imposing that all the eigenfunctions in the cluster can be approximated simultaneously (on the same subdivision) with rate $\sigma$.  It is equivalent to requiring that each eigenfunction belongs to 
$$
\mathcal A_\sigma:= \left\lbrace v  \in H_0^1(\Omega):  | v |_{\mathcal A_\sigma}:= \sup_{m \in \mathbb{N}} m^\sigma \inf_{\T \in \mathbb{T}, \#\T -\# \T_0 \le m} \inf_{\phi \in \mathbb V(\mathcal T)} ||| v - \phi |||  <\infty \right\rbrace, 
$$ 
but the norm equivalence constant depends on $J$.
Our choice of approximation class guarantees that the constant in estimate \eqref{eq3-19} is independent of $J$. 
\end{remark}
}

\begin{remark}[Equivalent approximation classes] 
In \cite{Gal15} optimality is expressed with respect to the class
$$
\mathcal B_\sigma := \left\lbrace v \in H_0^1(\Omega): |v|_{\mathcal  B_\sigma}:= \sup_{m \in \mathbb{N}} m^\sigma \inf_{\T \in \mathbb{T}, \#\T -\# \T_0 \le m} \| \nabla v - P_{\mathbb D(\mathcal T)} \nabla v \|   <\infty \right\rbrace, 
$$
where $P_\mathbb{D(\mathcal T)}$ is the $L_2$-projection onto $\mathbb{D}(\mathcal T)$, the space of piecewise polynomials of degree $r-1$ subordinate to $\mathcal T$.  Employing this class requires  proving equivalence between 
$\mathcal A_\sigma$ \modif{(see Remark~\ref{r:cluster})} and $\mathcal B_\sigma$, that is, approximation by functions in $\mathbb{D}(\mathcal T)$ is equivalent to approximation by gradients of functions in $\V(\mathcal T)$.  
%This is guaranteed by Proposition 5.2 of \cite{BN10}. 
In \cite{Gal15}, the needed result is obtained by citing \cite{CPS12, Gud10}, in which the equivalence is shown up to data oscillation terms.  However, the necessary equivalence has recently been shown to hold on arbitrary meshes without data oscillation terms in \cite{Vee15}, which could simplify the proof of Proposition 3.1 of \cite{Gal15}.  
\end{remark}

\begin{remark}[{Robustness with respect to the cluster size}]\label{r:theta}
{The dependence of $H_0$ and $\theta$ on various quantities is given in detail in \cite{Gal15}.   The threshold condition for $\theta$ depends on the second equivalence constant in \eqref{eq3-10}.  Thus \cite{Gal15} requires $0 < \theta \le [C_2 (\lambda_{n+N}/\lambda_{n+1})^4(2N^2+4N^3)]^{-1}$ (cf. Lemma 7.3), where $C_2$ is independent of essential quantities (especially cluster properties).  We only require $0 < \theta \le 3 C_2^{-1}$, thus establishing that rate-optimal adaptive convergence holds with the marking parameter chosen independent of the cluster properties.   The analysis of \cite{Gal15} however still indicates that the asymptotic nature of the optimality result may be more or less pronounced depending on cluster properties (cf. \eqref{eq3-30}).  
} \end{remark}

Finally, optimal approximation of eigenvalues follows as in \cite[Corollary 3.1]{Gal15}. 

\begin{corollary}[Quasi-Optimal Approximation of Eigenvalues clusters]\label{c:opt}
Under the conditions of Theorem~\ref{t:opt} there exists $C_{n,N}>0$ such that for $n+1 \le j \le n+N$, 
$$
| \lambda_{\ell,j}-\lambda_j |   \le C_{n,N} (\# \T_{\ell} - \# \T_0)^{-2\sigma}\sum_{j \in J} |u_j|_{\mathcal A_\sigma}^2.
$$
\end{corollary} 

Note that the constant $C_{n,N}$ above depends on properties of the eigenvalue cluster; cf. \cite[Corollary 3.1]{Gal15} for a precise expression.

%------------------------------------------------------------------------------------------------------------------------------------------ 
\section{Numerical experiments} \label{sec:numerics} 
%------------------------------------------------------------------------------------------------------------------------------------------

To illustrate our result, consider the slit domain 
 $$
 \Omega := (-1,1)^2 \setminus \left(\begin{array}{l} 
 \mathrm{conv}\lbrace (0.5,0),(1,0)\rbrace \cup  \mathrm{conv}\lbrace (0,0.5),(0,1)\rbrace \cup  \\
 \mathrm{conv}\lbrace (-0.5,0),(-1,0)\rbrace \cup  \mathrm{conv}\lbrace (0,-0.5),(0,-1)\rbrace 
 \end{array} \right)
$$
subdivided into $16$ squares, each of diameter $0.5$. 
The finite element spaces consist of continuous piecewise polynomials of degree $r=1$, $2$, or $3$ in each coordinate direction.
The adaptive algorithm is implemented within the \emph{deal.ii} \cite{dealii} library.  We take $n=0$ and consider the cases $N=4, 12$.  Our highest-precision computations indicate that the first twelve eigenvalues (truncated to 6 decimal places) are 10.147392, 17.662596, 17.662596, 19.739208, 26.101811, 38.349159, 38.349159, 46.553966, 49.149607, 49.149607, 49.348022, and 49.348022. In Figure~\ref{f:estimators}, we plot the error estimators $\eta_l:=\left(\sum_T\eta_l^2(T)\right)^{1/2}$ against the total number of degrees of freedom for different values of marking parameters $\theta$ and cluster sizes $N$. 

We do not know the exact eigenfunctions and so cannot plot actual errors. However, in the case of the Laplacian \cite{GMZ09} guarantees plain convergence (without rates) of AFEM for simple and multiple eigenvalues to the corresponding continuous eigenpairs.  These results rely on completely differently proof techniques than do ours, and adaptation to the case of clustered eigenvalues is straightforward.  The analysis of \cite{GMZ09} also guarantees that $\max_{T \in \T_\ell} h_T \rightarrow 0$ as $\ell \rightarrow \infty$, which yields reliability of $\eta_\ell$ for $\ell$ sufficiently large (cf. \eqref{eq3-30}).  It is thus meaningful to track $\eta_\ell$ instead of actual errors.  In addition, employing a cluster-independent marking parameter as suggested by our theory is reasonable also in the pre-asymptotic range as the plain convergence analysis of \cite{GMZ09} guarantees that the algorithm will eventually reach the asymptotic range.  

Optimal decay rates of $-r/2$ are observed provided $\theta\le .95$ when $r=2,3$ and $\theta \le 0.9$ when $r=1$.
According to Corollary~\ref{c:opt} ($\sigma = 1/2$), this indicates an optimal decay rate of $-r$ for the error in approximating each eigenvalue in the cluster.
As guaranteed by our analysis, the range of $\theta$ for which optimality is recovered is not affected by the cluster properties.  In contrast, following \cite{Gal15} as explained in Remark \ref{r:theta} leads to the restriction $\theta \le [C_2 (\lambda_{n+N}/\lambda_{n+1})^4 (2N^2+4N^3)]^{-1}$ with $C_2$ independent of essential quantities.  For our particular computations, this yields:
\begin{align*}
\begin{aligned}
\label{theta:calc}
\theta & \le  C_2^{-1} \left ( \frac{19.739208}{10.147392} \right )^{-4} \frac{1}{2 \cdot 4^2 + 4 \cdot 4^3} \approx  2.42 \times 10^{-4}  C_2^{-1}, ~~~n=0, ~N=4,
\\ \theta & \le C_2^{-1} \left ( \frac{49.348022}{10.147392} \right ) ^{-4} \frac{1}{2 \cdot 12^2 + 4 \cdot 12^3} \approx 2.48 \times 10^{-7} C_2^{-1}, ~~~n=0, ~N=12.
\end{aligned}
\end{align*}
Thus following precisely the theory of \cite{Gal15} would lead to a thousand-fold reduction in $\theta$ when moving from our first to our second computational example.  This would potentially require a massive increase in the number of AFEM iterations required in order to achieve a given error reduction.  We have demonstrated theoretically and confirmed computationally that this increase in computational expense is unnecessary.

\begin{figure}
\begin{tabular}{ccc}
\includegraphics[width=0.3\textwidth]{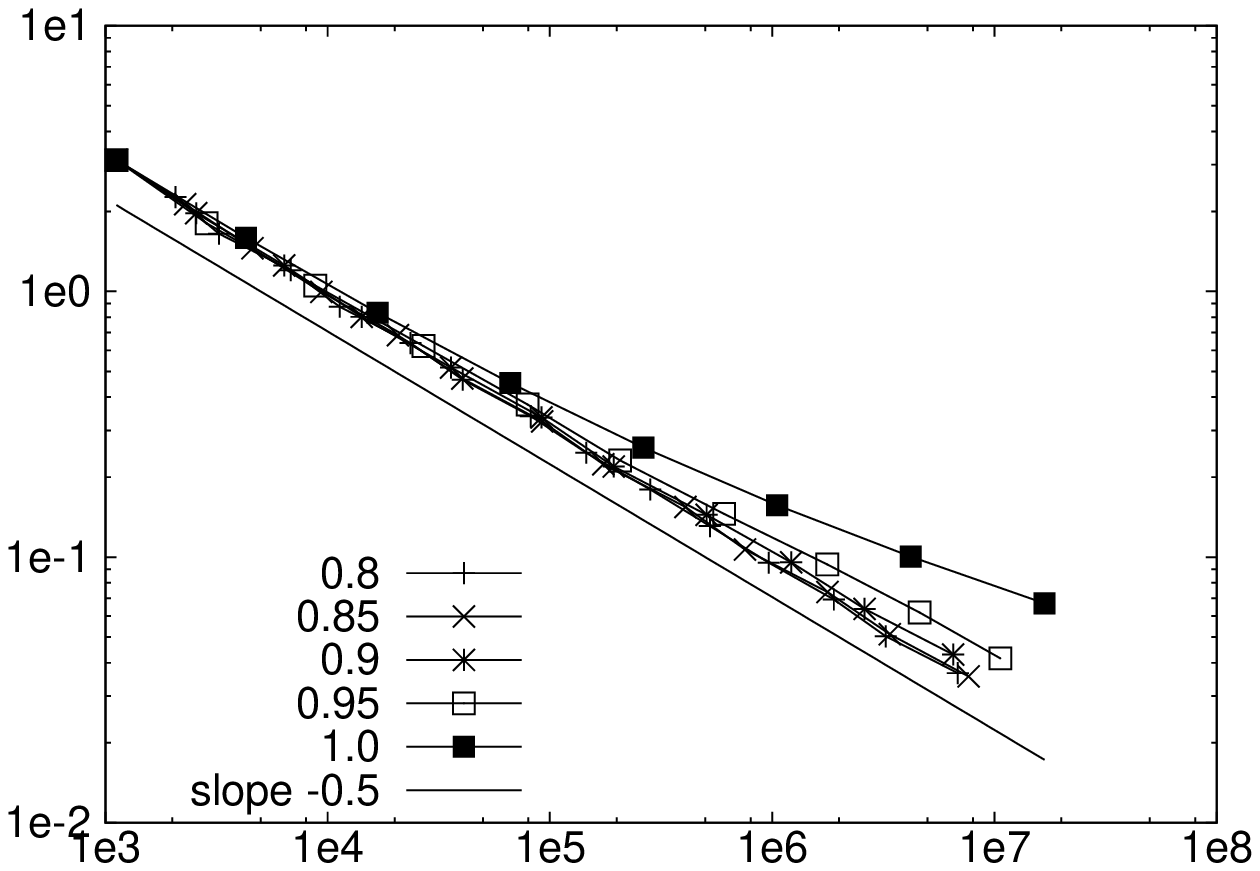} &
\includegraphics[width=0.3\textwidth]{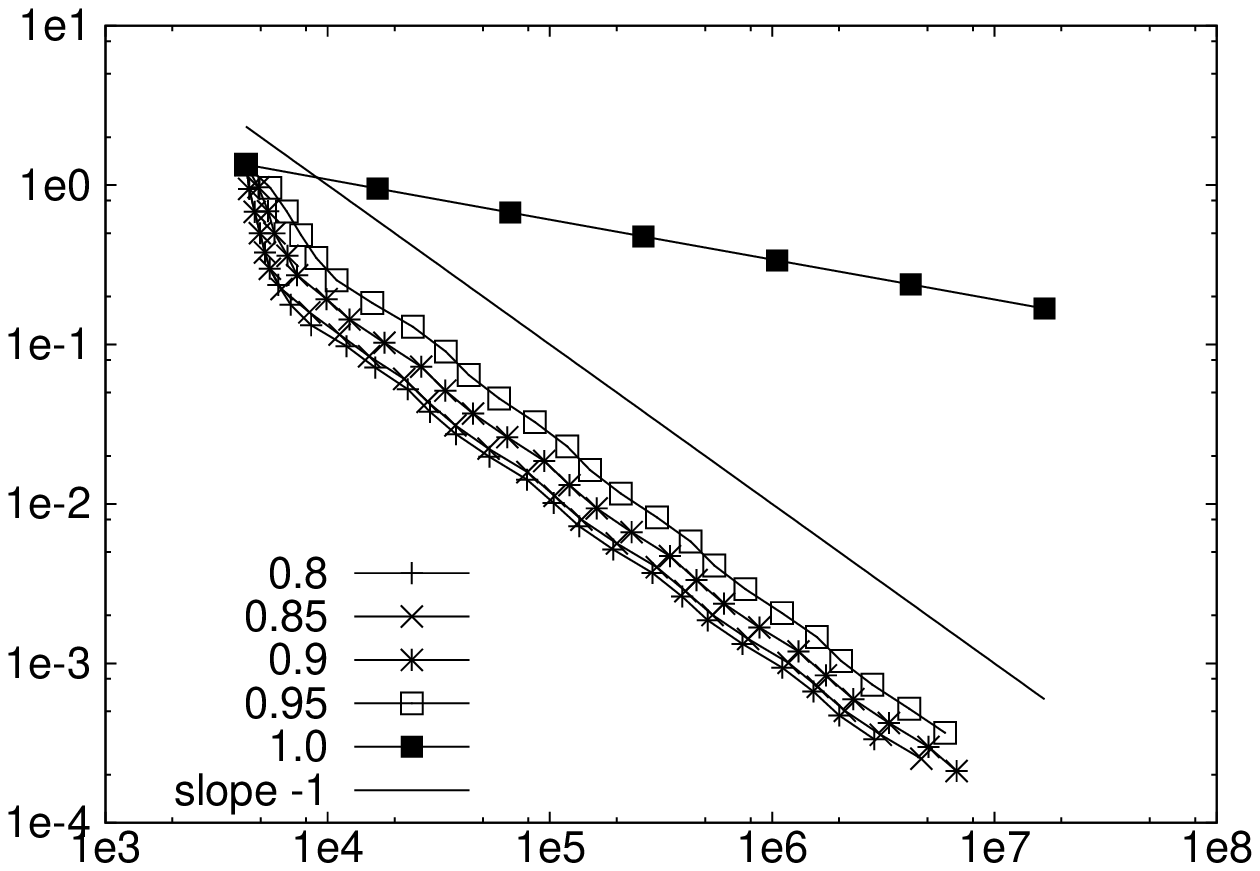} &
\includegraphics[width=0.3\textwidth]{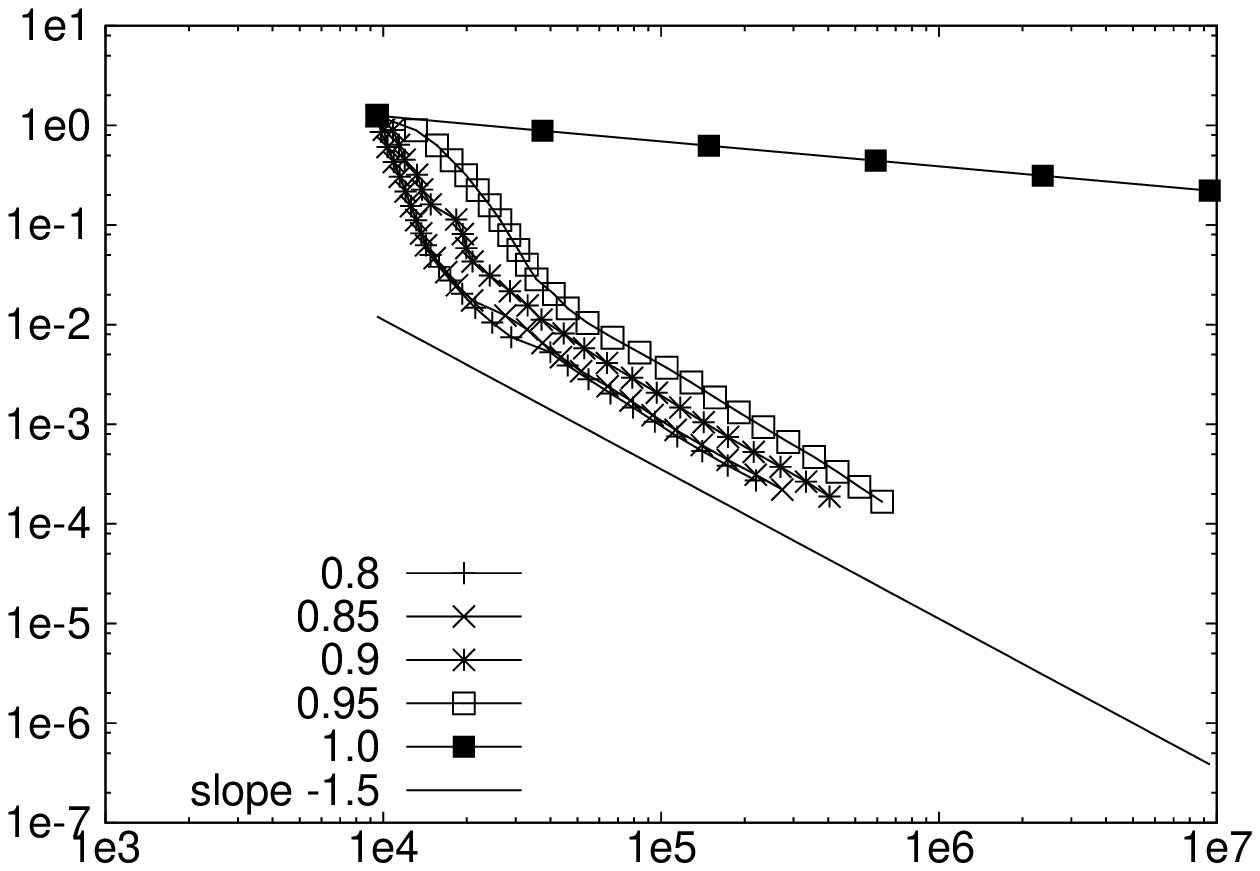} \\
\includegraphics[width=0.3\textwidth]{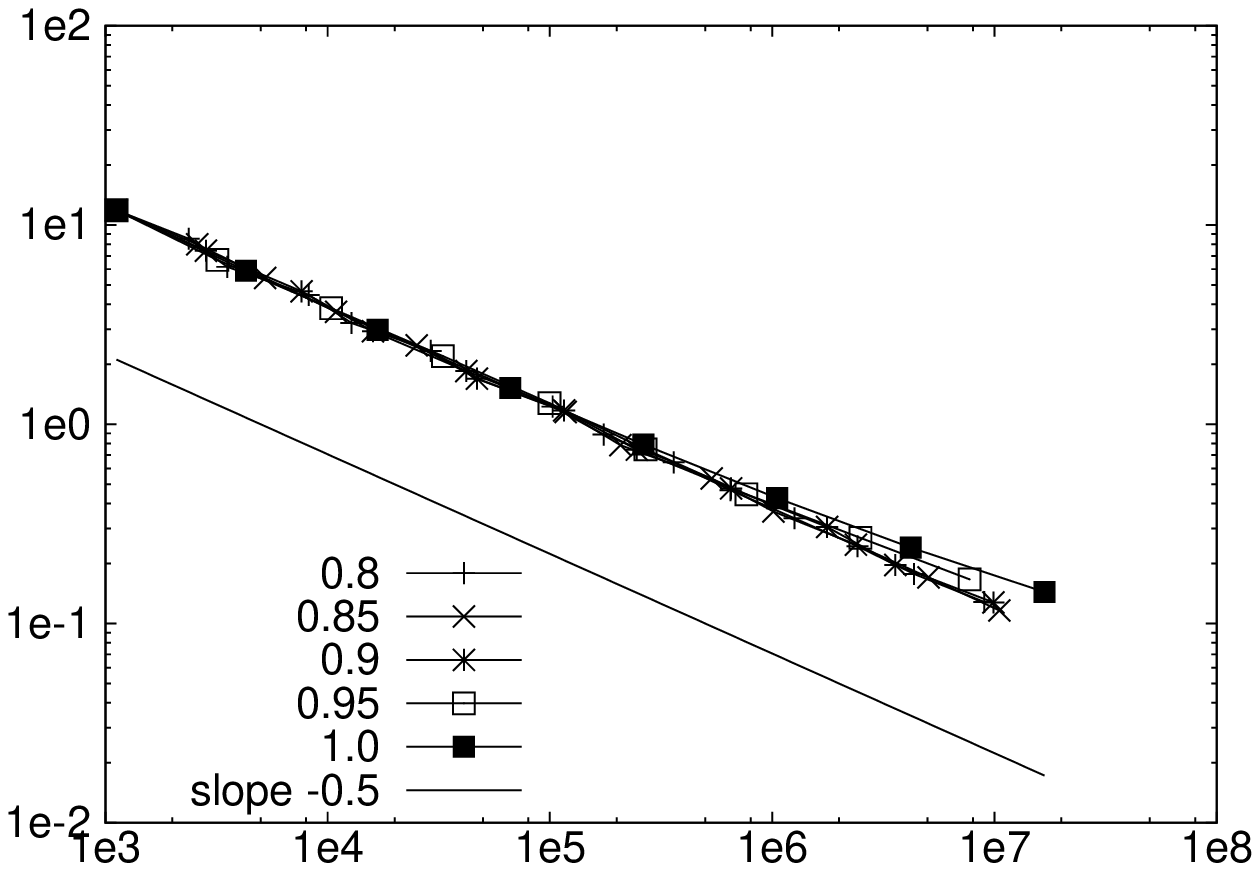} &
\includegraphics[width=0.3\textwidth]{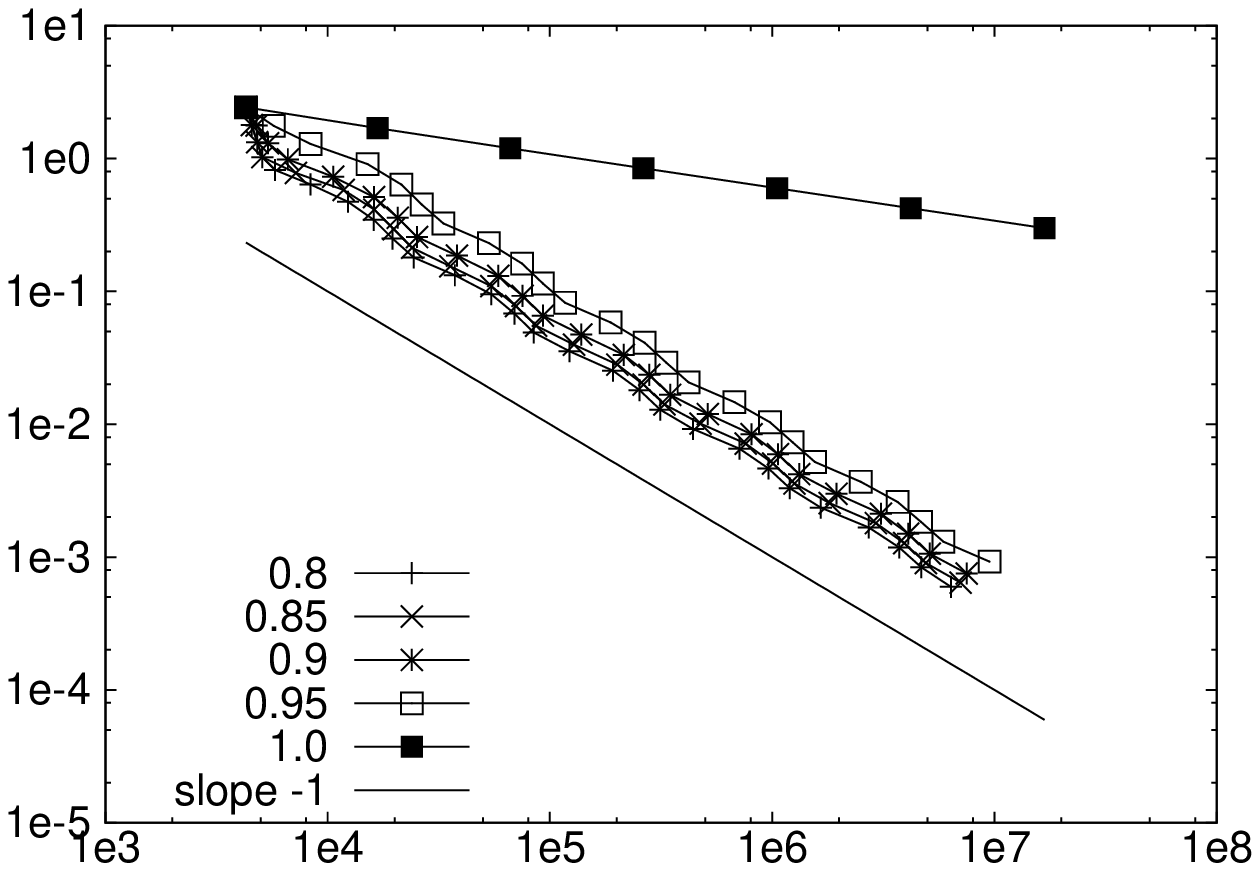} &
\includegraphics[width=0.3\textwidth]{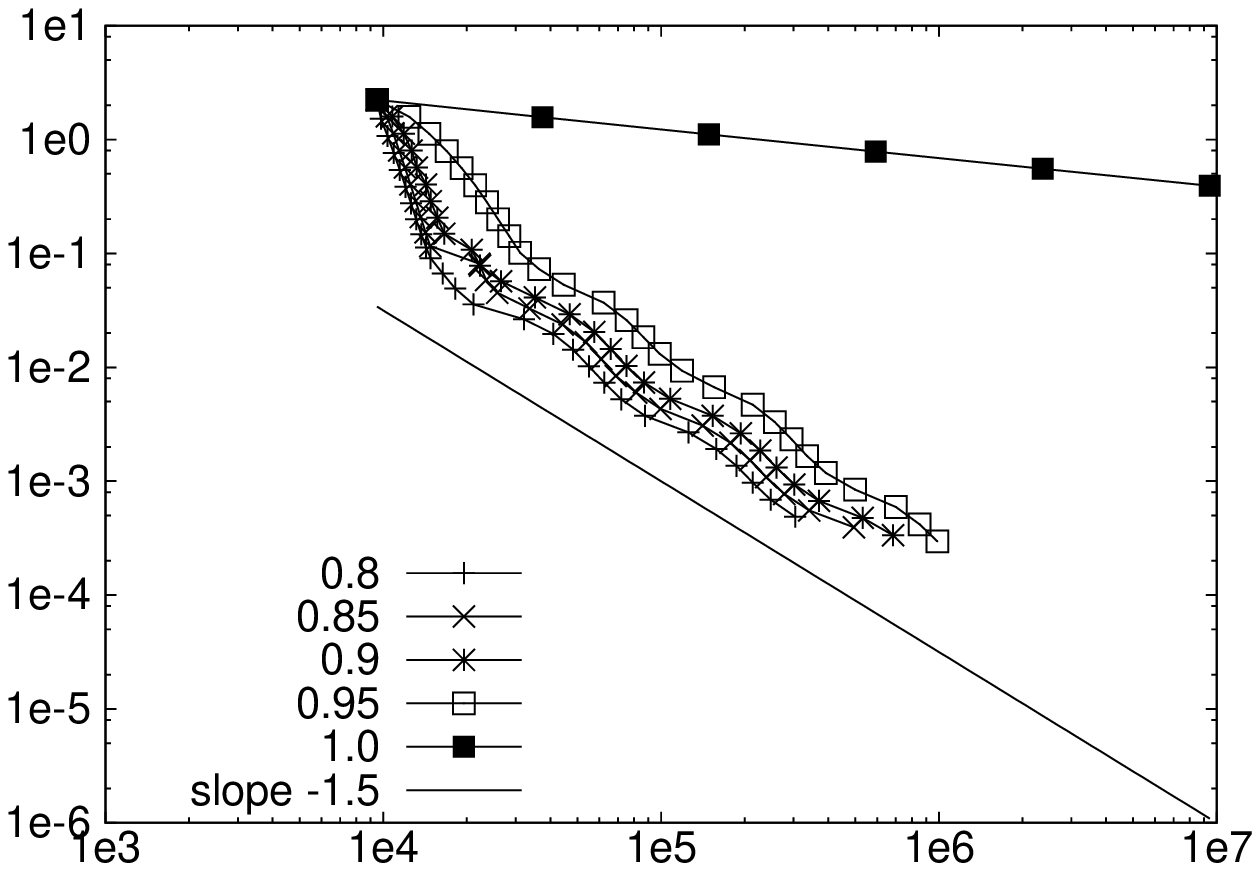} \\
\end{tabular}
\caption{Values of $\eta_l$ versus $\textrm{dim}(\mathbb V_\ell)$ during the adaptive process for marking parameters $\theta=0.8, 0.85, 0.9,0.95,1$, with $n=0$, $N=4$ (top row) and  $n=0$, $N=12$ (bottom row) in each cases using continuous piecewise polynomials of degree $r$ (column $r$). 
The optimal decay rate of $-r/2$ is observed as soon as $\theta<1$ when $r=2,3$ and $\theta <0.95$ when $r=1$.}\label{f:estimators}
\end{figure}

\bibliographystyle{siam}

\end{document}